\documentclass[12 pt]{article}
\usepackage{stmaryrd}
\usepackage{times}
\usepackage{booktabs}
\usepackage{subfigure}

\usepackage{pifont}
\usepackage{floatrow}
\usepackage{caption}
\usepackage{mathrsfs}
\usepackage[fleqn]{amsmath}
\usepackage{amsfonts,amsthm,amssymb,mathrsfs,bbding}
\usepackage{txfonts}
\usepackage{graphics,multicol}
\usepackage{graphicx}
\usepackage{color}
\usepackage{caption}
\usepackage{indentfirst}
\usepackage{cite}
\usepackage{latexsym,bm}
\usepackage{enumerate}
\evensidemargin=\oddsidemargin\topmargin=-1.5cm

\newtheorem{lem}{Lemma}[section]

\newtheorem{thm}{Theorem}[section]
\newtheorem{cor}{Corollary}[section]

\newtheorem{remark}{Remark}

\theoremstyle{definition}

\addtocounter{section}{0}
\begin{document}
\title{On graphs with $m(\partial^L_1)=n-3$ \footnote{Supported
by the National Natural Science Foundation of China (Grant Nos. 11671344, 11531011, 11626205).}}
\author{{\small Lu Lu, \ \ Qiongxiang Huang\footnote{
Corresponding author.
Email: huangqx@xju.edu.cn},\ \ Xueyi Huang}\\[2mm]\scriptsize
College of Mathematics and Systems Science,
\scriptsize Xinjiang University, Urumqi, Xinjiang 830046, P.R. China}
\date{}
\maketitle {\flushleft\large\bf Abstract } Let $\partial^L_1\ge\partial^L_2\ge\cdots\ge\partial^L_n$ be the distance Laplacian eigenvalues of a connected graph $G$ and $m(\partial^L_i)$ the multiplicity of $\partial^L_i$. It is well known that the graphs with $m(\partial^L_1)=n-1$ are complete graphs. Recently, the graphs with $m(\partial^L_1)=n-2$ have been characterized by Celso et al. In this paper, we completely determine the graphs with $m(\partial^L_1)=n-3$.
\begin{flushleft}
\textbf{Keywords:} Distance Laplacian matrix; Laplacian matrix; Largest eigenvalue; Determined by distance Laplacian spectrum
\end{flushleft}
\textbf{AMS subject classifications:} 05C50; 05C12; 15A18

\section{Introduction}\label{se-1}
In this paper we only consider simple connected graphs. Let $G=(V,E)$ be a connected graph with vertex set $V=\{v_1,v_2,\ldots,v_n\}$ and edge set $E=\{e_1,e_2,\ldots,e_m\}$. The \emph{distance} between $v_i$ and $v_j$, denoted by $d_G(v_i,v_j)$, is defined as the length of a shortest path between them. The \emph{diameter} of $G$, denoted by $d(G)$, is the maximum distance between any two vertices of $G$. The \emph{distance matrix} of $G$, denoted by $\mathcal{D}(G)$, is the $n\times n$ matrix whose $(i,j)$-entry is equal to $d_G(v_i,v_j)$, $i,j=1,2,\ldots,n$. The \emph{transmission} $Tr(v_i)$ of a vertex $v_i$ is defined as the sum of the distances between $v_i$ and all other vertices in $G$, that is, $Tr(v_i)=\sum_{j=1}^n d_G(v_i,v_j)$. For more details about the distance matrix we refer the readers to \cite{Aouchiche}. Aouchiche and Hansen \cite{Aouchiche-2} introduced the Laplacian for the distance matrix of $G$ as $\mathcal{D}^L(G)=Tr(G)-\mathcal{D}(G)$, where $Tr(G)=diag(Tr(v_1),Tr(v_2),\ldots,Tr(v_n))$ is the diagonal matrix of the vertex transmissions in $G$. The eigenvalues of $\mathcal{D}^L(G)$, listed by $\partial^L_1\ge\partial^L_2\ge\cdots\ge\partial_n=0$, are called the \emph{distance Laplacian eigenvalues} of $G$. The multiplicity of $\partial^L_i$ is denoted by $m(\partial^L_i)$. The distance eigenvalues together with their multiplicities is called the \emph{distance Laplacian spectrum} of $G$, denoted by $\emph{Spec}_{\mathcal{L}}(G)$.

The distance Laplacian matrix aroused many active studies, such as \cite{Tian,Aouchiche,Nath,Celso}. Graphs with few distinct eigenvalues form an interesting class of graphs and possess nice combinatorial properties. With respect to distance Laplacian eigenvalues, we focus on the graphs with $m(\partial^L_1)$ being large. Denote by $\mathcal{G}(n)$ the set of connected graphs of order $n$. Let $\mathcal{G}(n,k)=\{G\in\mathcal{G}(n)\mid m(\partial^L_1)=k\}$ be the set of connected graphs with $m(\partial^L_1)=k$. Aouchiche and Hansen \cite{Aouchiche} proved that $\mathcal{G}(n,n-1)=\{K_n\}$ and conjectured that $\mathcal{G}(n,n-2)=\{K_{1,n-1},K_{n/2,n/2}\}$, which has been confirmed by Celso et al. \cite{Celso}. Motivated by their work, we try to characterize $\mathcal{G}(n,n-3)$. In this paper, we completely determine the graphs in $\mathcal{G}(n,n-3)$ (Theorem \ref{thm-3-3}). By the way, we show that all these graphs are determined by their distance Laplacian spectra (Corollary \ref{cor-3-3}).

\section{Preliminaries}\label{se-2}
Let $G$ be a connected graph, we always denote by $N_G(v)$ the neighbuor set of $v$ in $G$, that is, $N_G(v)=\{u\in V(G)\mid u\sim v\}$. The $i$-th largest distance Laplacian eigenvalue of $G$ is denoted by $\partial^L_i(G)$, whose multiplicity is denoted by $m(\partial^L_i(G))$. When it is clear from the context which graph $G$ we mean, we delete $G$ from the notations like $d_G(v_i,v_j)$, $N_G(v)$, $\partial^L_i(G)$ and $m(\partial^L_i(G))$. For a subset $S\subseteq V(G)$, let $G[S]$  denote the subgraph of $G$ induced by $S$.

As usual, we always write, respectively, $K_n$, $P_n$ and $C_n$ for the complete graph, the path and the cycle on $n$ vertices. For integers $a_1,a_2,\ldots,a_k\ge1$, let $K_{a_1,a_2,\ldots,a_k}$ denote the complete $k$-partite graph on $a_1+a_2+\cdots+a_k$ vertices. Let $G$ be a connected graph, denote by $\bar{G}$ the \emph{complement} of $G$, which is a graph with vertex set $V(\bar{G})=V(G)$ and two vertices are adjacent whenever they are not adjacent in $G$. Let $G_1=(V_1,E_1)$ and $G_2=(V_2,E_2)$ be two connected graphs, the \emph{(disjoint-)union} of $G_1$ and $G_2$ is the graph $G_1\cup G_2$, whose vertex set is $V_1\cup V_2$ and edge set is $E_1\cup E_2$. The \emph{join} of $G_1$ and $G_2$ is the graph $G_1\nabla G_2$, which is obtained from $G_1\cup G_2$ by joining each vertex of $G_1$ with every vertex of $G_2$. Moreover, we write $mG=\underbrace{G\cup G\cup \cdots\cup G}_m$ for an integer $m\ge2$.

At first, we introduce the famous Cauchy interlacing theorem.
\begin{thm}[\cite{Horn}]\label{thm-2-1}
Let $A$ be a real symmetric matrix of order $n$ with eigenvalues $\lambda_1(A)\ge\lambda_2(A)\ge\cdots\ge\lambda_n(A)$ and let $M$ be a principal submatrix of $A$ with order $m\le n$ and eigenvalues $\lambda_1(M)\ge\lambda_2(M)\ge\cdots\ge\lambda_m(M)$. Then $\lambda_i(A)\ge\lambda_i(M)\ge\lambda_{n-m+i}(A)$, for all $1\le i\le m$.
\end{thm}
Let $G$ be a graph on $n$ vertices, denote by $\mu_1\ge\mu_2\ge\cdots\ge\mu_{n-1}\ge\mu_n=0$ the Laplacian eigenvalues of $G$ and $m(\mu_i)$ the multiplicity of $\mu_i$. There are many pretty properties for Laplacian eigenvalues.
\begin{lem}[\cite{Cvet}]\label{lem-2-1}
Let $G$ be a graph on $n$ vertices with Laplacian eigenvalues $\mu_1\ge\mu_2\ge\cdots\ge\mu_{n-1}\ge\mu_n=0$. Then we have the following results.\\
(i) Denote by $m(0)$ the multiplicity of $0$ as a Laplacian eigenvalue and $w(G)$ the number of connected components of $G$. Then $w(G)=m(0)$.\\
(ii) $G$ has exactly two distinct Laplacian eigenvalues if and only if $G$ is a union of complete graphs of the same order and isolate vertices.\\
(iii) The Laplacian eigenvalues of $\bar{G}$ are given by $\mu_i(\bar{G})=n-\mu_{n-i}$ for $i=1,2,\ldots,n-1$ and $\mu_n(\bar{G})=0$.\\
(iv) Let $H$ be a graph on $m$ vertices with Laplacian eigenvalues $\mu_1'\ge\mu_2'\ge\cdots\ge\mu_m'=0$, then the Laplacian spectrum of $G\nabla H$ is \[\{n+m,m+\mu_1,m+\mu_2,\ldots,m+\mu_n,n+\mu_1',n+\mu_2',
\ldots,n+\mu_m',0\}.\]
\end{lem}
With respect to distance Laplacian eigenvalues, there are some similar results. The following results are given by Aouchiche and Hansen.
\begin{thm}[\cite{Aouchiche-2}]\label{thm-2-2}
Let $G$ be a connected graph on $n$ vertices with $d(G)\le 2$. Let $\mu_1\ge\mu_2\ge\cdots\ge\mu_{n-1}\ge\mu_n=0$ be the Laplacian spectrum of $G$. Then the distance Laplacian spectrum of $G$ is $2n-\mu_{n-1}\ge 2n-\mu_{n-2}\ge\cdots\ge2n-\mu_1\ge\partial^L_n=0$. Moreover, for every $i\in\{1, 2,\ldots, n-1\}$ the eigenspaces corresponding to $\mu_i$ and to $2n-\mu_i$ are the same.
\end{thm}
\begin{thm}[\cite{Aouchiche-2}]\label{thm-2-3} Let $G$ be a connected graph on $n$ vertices. Then $\partial^L_{n-1}\ge n$ and $\partial^L_{n-1}=n$ if
and only if $\bar{G}$ is disconnected. Furthermore, the multiplicity of $n$ as a distance Laplacian eigenvalue is one less than the number of connected components of $\bar{G}$.
\end{thm}
\begin{thm}[\cite{Aouchiche-2}]\label{thm-2-4} Let $G$ be a connected graph on $n$ vertices and $m\ge n$ edges. Consider the connected graph $G'$ obtained from $G$ by the deletion of an edge. Let $\partial^L_1,\partial^L_2,\ldots,\partial^L_n$ and $\partial'^L_1,\partial'^L_2,\ldots,\partial'^L_n$ denote the distance Laplacian spectra of $G$ and $G'$ respectively. Then $\partial'^L_i\ge \partial^L_i$ for all $i=1,\ldots,n$.
\end{thm}
A graph $G$ is said to be a \emph{cograph} if it contains no induced $P_4$. There's a pretty result about cographs.
\begin{lem}[\cite{Corneil}]\label{lem-2-2}
Given a graph $G$, the following statements are equivalent:\\
1) $G$ is a cograph.\\
2) The complement of any connected subgraph of $G$ with at least two vertices is
disconnected.\\
3) Every connected induced subgraph of $G$ has diameter less than or equal to $2$.
\end{lem}
\section{Main results}\label{se-3}
Recall that $\mathcal{G}(n,k)=\{G\in \mathcal{G}(n)\mid m(\partial^L_1)=k\}$. Aouchiche and Hansen \cite{Aouchiche} proved that $\mathcal{G}(n,n-1)=\{K_n\}$. Recently, Celso et al. \cite{Celso} proved that $\mathcal{G}(n,n-2)=\{K_{1,n-1},K_{n/2,n/2}\}$. They also made efforts to characterize $\mathcal{G}(n,n-3)$. Though they did not give a complete characterization, their ideas are enlightening. Especially, they proved that the graphs in $\mathcal{G}(n,n-3)$ contain no induced $P_5$.
\begin{lem}[\cite{Celso}, Theorem 4.1]\label{lem-3-1}
Let $G\in \mathcal{G}(n,n-3)$ with $n\ge 5$ then $G$ does not contain induced $P_5$.
\end{lem}
\begin{remark}\label{re-1}
If $G$ does not contain induced $P_5$, then $d(G)\le 3$. Note that $K_n\not\in\mathcal{G}(n,n-3)$. We obtain that $d(G)=2$ or $d(G)=3$ for any graph $G\in\mathcal{G}(n,n-3)$ with $n\ge5$.
\end{remark}
\begin{lem}[\cite{Celso}, Theorem 3.3]\label{lem-3-2}
If $G$ is a connected graph then $\partial^L_1\ge\max_{v\in V(G)}Tr(v)+1$ with equality holds if and only if $G\cong K_n$.
\end{lem}
\begin{lem}\label{lem-3-3}
Let $G\in\mathcal{G}(n,n-3)$ with $n\ge 6$, then $\partial^L_1$ is integral.
\end{lem}
\begin{proof}
Let $f(x)$ be the characteristic polynomial of $\mathcal{D}^L(G)$. As $\mathcal{D}^L(G)$ only contains integral entries, we obtain that $f(x)$ is a monic polynomial with integral coefficients. Let $p(x)$ be the minimal polynomial of $\partial^L_1$, then $p(x)\in Z[x]$ is irreducible in $Q[x]$ and $(p(x))^{n-3}|f(x)$. We assume that $p(x)$ is a polynomial of degree at least $2$. Therefore, $p(x)$ has another root $\partial\ne0$, which is also a distance Laplacian eigenvalue of $G$ with multiplicity $n-3$. It leads to that $n\le2(n-3)\le n-1$, a contradiction. Thus, we have $p(x)=x-\partial^L_1$ and the result follows.
\end{proof}
From Lemmas \ref{lem-3-2} and \ref{lem-3-3}, we get the following result.
\begin{cor}\label{cor-3-1}
Let $G=(V,E)\in \mathcal{G}(n,n-3)$ with $n\ge 6$, then we have $\partial^L_1\ge\max_{v\in V}Tr(v)+2$. Furthermore, if there exists $v_0\in V$ such that $\partial^L_1=Tr(v_0)+2$, then $Tr(v_0)=\max_{v\in V}Tr(v)$.
\end{cor}
\begin{proof}
Obviously, $G\ne K_n$. By Lemma \ref{lem-3-2}, we have that $\partial^L_1>\max_{v\in V}Tr(v)+1$. Besides, we get that $\partial^L_1$ is integral from Lemma \ref{lem-3-3}. Therefore, we have that $\partial^L_1\ge\max_{v\in V}Tr(v)+2$. Furthermore, if $\partial^L_1=Tr(v_0)+2$, then we have $Tr(v_0)+2\ge\max_{v\in V}Tr(v)+2$. It follows that $Tr(v_0)=\max_{v\in V}Tr(v)$.
\end{proof}
We say that a graph $G$ is \emph{$P_5$-free} if it does not contain induced $P_5$. From Lemma \ref{lem-3-1}, all graphs in $\mathcal{G}(n,n-3)$ are $P_5$-free. By Remark \ref{re-1}, a $P_5$-free graph may have diameter $2$ or $3$. Now we discuss $P_5$-free graphs with diameter $3$.
\begin{lem}\label{lem-3-4}
Let $G$ be a connected $P_5$-free graph on $n\ge 5$ vertices with $d(G)=3$. Then at least one of  $I_i$ for $i=1,2,3,4,5$ (shown in Fig. \ref{fig-1}) is an induced subgraph of $G$.
\end{lem}
\begin{proof}
Suppose that $d(v_1,v_4)=3$ and $P=v_1v_2v_3v_4$ is a shortest path from $v_1$ to $v_4$. Since $n\ge 5$ and $G$ is connected, there exists $u\in V(G)\setminus V(P)$ such that $N(u)\cap V(P)\ne\emptyset$, where $N(u)=\{v\in V(G)\mid v\sim u\}$ is the neighbour set of $u$ in $G$. Moreover, since $d(v_1,v_4)=3$, we have that $v_1$ and $v_4$ cannot be adjacent to $u$ simultaneously, that is, $\{v_1,v_4\}\not\subseteq N(u)$. Therefore, we have $1\le |N(u)\cap V(P)|\le 3$.

Assume that $|N(u)\cap V(P)|=1$. We claim that $N(u)\cap V(P)=\{v_2\}$ or $\{v_3\}$ since $G$ is $P_5$-free. Both of them lead to the induced subgraph $I_1$.

Assume that $|N(u)\cap V(P)|=2$. We claim that $N(u)\cap V(P)=\{v_1,v_2\}$, $\{ v_3,v_4\}$, $\{v_1,v_3\}$, $\{v_2,v_4\}$, or $\{v_2,v_3\}$ because $\{v_1,v_4\}\not\subseteq N(u)$. The former two cases lead to the induced subgraph $I_2$, the next two cases lead to the induced subgraph $I_3$ and the last case leads to the induced subgraph $I_4$.

Assume that $|N(u)\cap V(P)|=3$. We claim that $N(u)\cap V(P)=\{v_1,v_2,v_3\}$ or $\{ v_2,v_3,v_4\}$ because $\{v_1,v_4\}\not\subseteq N(u)$. Both cases lead to the induced subgraph $I_5$.
\end{proof}
\begin{figure}[htbp]
\begin{center}
\unitlength 5mm 
\linethickness{0.4pt}
\ifx\plotpoint\undefined\newsavebox{\plotpoint}\fi 
\begin{picture}(24,9)(0,0)
\thicklines
\put(1,8){\line(1,0){6}}
\put(9,8){\line(1,0){6}}
\put(17,8){\line(1,0){6}}
\put(3,8){\line(0,-1){2}}
\multiput(11,8)(-.0333333,-.0666667){30}{\line(0,-1){.0666667}}
\multiput(10,6)(-.0333333,.0666667){30}{\line(0,1){.0666667}}
\multiput(17,8)(.03333333,-.03333333){60}{\line(0,-1){.03333333}}
\multiput(19,6)(.03333333,.03333333){60}{\line(0,1){.03333333}}
\put(1,3){\line(1,0){6}}
\put(9,3){\line(1,0){6}}
\multiput(3,3)(.0333333,-.0666667){30}{\line(0,-1){.0666667}}
\multiput(4,1)(.0333333,.0666667){30}{\line(0,1){.0666667}}
\multiput(9,3)(.03333333,-.03333333){60}{\line(0,-1){.03333333}}
\put(11,1){\line(0,1){2}}
\multiput(13,3)(-.03333333,-.03333333){60}{\line(0,-1){.03333333}}
\put(1,8){\circle*{.5}}
\put(3,8){\circle*{.5}}
\put(5,8){\circle*{.5}}
\put(7,8){\circle*{.5}}
\put(3,6){\circle*{.5}}
\put(9,8){\circle*{.5}}
\put(11,8){\circle*{.5}}
\put(13,8){\circle*{.5}}
\put(15,8){\circle*{.5}}
\put(17,8){\circle*{.5}}
\put(19,8){\circle*{.5}}
\put(21,8){\circle*{.5}}
\put(23,8){\circle*{.5}}
\put(10,6){\circle*{.5}}
\put(19,6){\circle*{.5}}
\put(1,3){\circle*{.5}}
\put(3,3){\circle*{.5}}
\put(5,3){\circle*{.5}}
\put(7,3){\circle*{.5}}
\put(9,3){\circle*{.5}}
\put(11,3){\circle*{.5}}
\put(13,3){\circle*{.5}}
\put(15,3){\circle*{.5}}
\put(4,1){\circle*{.5}}
\put(11,1){\circle*{.5}}
\put(1,8.7){\makebox(0,0)[cc]{\footnotesize$v_1$}}
\put(3,8.7){\makebox(0,0)[cc]{\footnotesize$v_2$}}
\put(5,8.7){\makebox(0,0)[cc]{\footnotesize$v_3$}}
\put(7,8.7){\makebox(0,0)[cc]{\footnotesize$v_4$}}
\put(9,8.7){\makebox(0,0)[cc]{\footnotesize$v_1$}}
\put(11,8.7){\makebox(0,0)[cc]{\footnotesize$v_2$}}
\put(13,8.7){\makebox(0,0)[cc]{\footnotesize$v_3$}}
\put(15,8.7){\makebox(0,0)[cc]{\footnotesize$v_4$}}
\put(17,8.7){\makebox(0,0)[cc]{\footnotesize$v_1$}}
\put(19,8.7){\makebox(0,0)[cc]{\footnotesize$v_2$}}
\put(21,8.7){\makebox(0,0)[cc]{\footnotesize$v_3$}}
\put(23,8.7){\makebox(0,0)[cc]{\footnotesize$v_4$}}
\put(1,3.7){\makebox(0,0)[cc]{\footnotesize$v_1$}}
\put(3,3.7){\makebox(0,0)[cc]{\footnotesize$v_2$}}
\put(5,3.7){\makebox(0,0)[cc]{\footnotesize$v_3$}}
\put(7,3.7){\makebox(0,0)[cc]{\footnotesize$v_4$}}
\put(9,3.7){\makebox(0,0)[cc]{\footnotesize$v_1$}}
\put(11,3.7){\makebox(0,0)[cc]{\footnotesize$v_2$}}
\put(13,3.7){\makebox(0,0)[cc]{\footnotesize$v_3$}}
\put(15,3.7){\makebox(0,0)[cc]{\footnotesize$v_4$}}
\put(3.6,6){\makebox(0,0)[cc]{\small$u$}}
\put(10.6,6){\makebox(0,0)[cc]{\small$u$}}
\put(19.6,6){\makebox(0,0)[cc]{\small$u$}}
\put(11.6,1){\makebox(0,0)[cc]{\small$u$}}
\put(4.6,1){\makebox(0,0)[cc]{\small$u$}}
\put(4,5){\makebox(0,0)[cc]{\small$I_1$}}
\put(11,5){\makebox(0,0)[cc]{\small$I_2$}}
\put(20,5){\makebox(0,0)[cc]{\small$I_3$}}
\put(4,0){\makebox(0,0)[cc]{\small$I_4$}}
\put(12,0){\makebox(0,0)[cc]{\small$I_5$}}
\end{picture}
\caption{The graphs $I_1, I_2,\ldots,I_5$}\label{fig-1}
\end{center}
\end{figure}
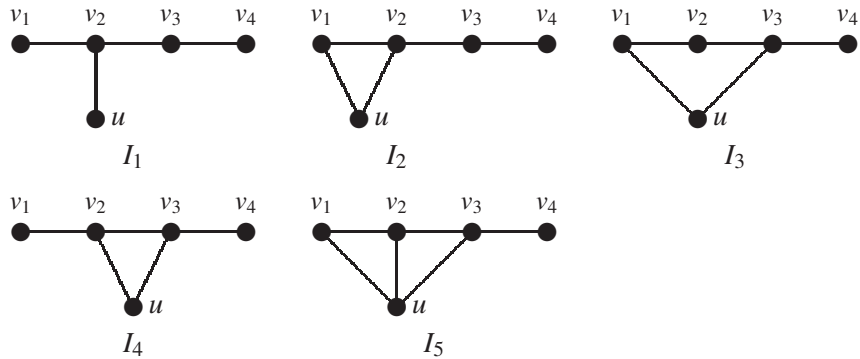
Next we introduce a tool which will be used frequently.
\begin{lem}\label{lem-3-5}
Let $G\in\mathcal{G}(n,n-3)$ with $n\ge 5$ and $M$ a principal submatrix of $\mathcal{D}^L(G)$ of order $5$. Then $\partial^L_1$ is also an eigenvalue of $M$ with multiplicity at least two. Furthermore, for each $1\le k\le 5$, there exists an eigenvector $z=(z_1,z_2,\ldots,z_5)^T$ of $M$ with respect to $\partial^L_1$ such that $z_k=0$ and $\sum_{i=1}^5z_i=0$.
\end{lem}
\begin{proof}
Let $\lambda_1\ge\lambda_2\ge\cdots\ge\lambda_5$ be the eigenvalues of $M$. By Theorem \ref{thm-2-1}, we have $\partial^L_1=\partial^L_{n-4}\le\lambda_1\le\partial^L_1$ and $\partial^L_1=\partial^L_{n-3}\le\lambda_2\le \partial^L_2=\partial^L_1$. Therefore, we have $\lambda_1=\lambda_2=\partial^L_1$. Suppose that $x=(x_1,\ldots,x_5)^T$ and $y=(y_1,\ldots,y_5)^T$ are two independent eigenvectors of $M$ with respect to $\partial^L_1$. For each fixed integer $1\le k\le 5$, by linear combination of $x$ and $y$, we get the eigenvector $z=(z_1,\ldots,z_5)^T$ satisfying $z_k=0$. Let  $z*=(z_1,\ldots,z_5,0,\ldots,0)^T$. Note that $\partial^L_1\ge \frac{{z*}^T\mathcal{D}^L(G)z*}{{z*}^Tz*}=\frac{z^TMz}{z^Tz}=\partial^L_1$. We get that $z*$ is an eigenvector of $\mathcal{D}^L(G)$ with respect to $\partial^L_1\ne 0$. Note that the all-ones vector $j$ is an eigenvector of $\mathcal{D}^L(G)$ with respect to $0$. We have ${z*}^Tj=\sum_{i=1}^5z_i=0$.
\end{proof}

Denote by $J(a,b)$ the graph obtained from $K_{1,a}\cup K_{1,b}$ by joining each pendent vertex of $K_{1,a}$ with every pendent vertex of $K_{1,b}$ (shown in Fig. \ref{fig-3}). The non-pendent vertices of $K_{1,a}$ and $K_{1,b}$ are called the \emph{roots} of $J(a,b)$.
\begin{lem}\label{lem-x-3-6}
Let $G$ be a connected $P_5$-free graph on $n\ge 5$ vertices with diameter $d(G)=3$. If none of $I_1$, $I_2$, $I_4$ and $I_5$ is an induced subgraph of $G$, then $G=J(a,b)$ for some positive integers $a,b\ge1$ and $a+b+2=n$.
\end{lem}
\begin{proof}
Let $d(v_1,v_4)=3$ and $P=v_1v_2v_3v_4$ a shortest path between $v_1$ and $v_4$. By Lemma \ref{lem-3-4}, at least one of $I_i$ (shown in Fig. \ref{fig-1}) is an induced subgraph of $G$ for $i=1,2,\ldots,5$. Since none of $I_1$, $I_2$, $I_4$ or $I_5$ is an induced subgraph of $G$, we obtain that $G$ contains induced $I_3$.

Note that $I_3=J(2,1)$ with roots $v_1$ and $v_4$ is an induced subgraph of $G$. We may assume that $G'=J(a,b)$ with roots $v_1$ and $v_4$ is the maximal  induced subgraph of $G$ including $J(2,1)$. Denote by $U_1=N_{G'}(v_1)=\{x_1,x_2,\ldots,x_a\}$   and $U_2=N_{G'}(v_4)=\{y_1,y_2,\ldots,y_b\}$. Obviously, $v_2,u\in U_1$ and $v_3\in U_2$.
In what follows we will show  that $G=J(a,b)$ with roots $v_1$ and $v_4$.

By the way of contradiction, assume that $G\not=J(a,b)$. Then there exists $v\in V(G)\setminus V(G')$ such that $N_G(v)\cap V(G')\ne\emptyset$.
Since $d(v_1,v_4)=3$,  $v$ is adjacent to at most one of $v_1$ and $v_4$. We claim that $v$ is exactly adjacent to one of  $v_1$ and $v_4$.  Otherwise, we have $v\not\sim v_1,v_4$. Then  $N_G(v)\cap U_1\ne \emptyset$ or $N_G(v)\cap U_2\ne\emptyset$. If $v$ is adjacent to some vertex in $N_G(v)\cap U_1$ and some vertex in $N_G(v)\cap U_2$, say $v\sim x_1$ and $v\sim y_1$ (see Fig. \ref{fig-3} (1)), then we get the induced subgraph $G[v_1,x_1,y_1,v_4,v]=I_4$, a contradiction. If $v$ is only adjacent to some vertex in $N_G(v)\cap U_1$, say $v\sim x_1$ (see Fig. \ref{fig-3} (2)), then we get the induced subgraph $G[v_1,x_1,y_1,v_4,v]=I_1$, a contradiction. If $v$ is only adjacent to some vertex in $N_G(v)\cap U_2$, say $v\sim y_1$,  then we also get the induced subgraph $G[v_1,x_1,y_1,v_4,v]=I_1$, a contradiction. Now we need to consider the following two situations.

{\flushleft\bf Case 1.} $v\sim v_1$ and $v\not\sim v_4$;

First, we will show that  $U_2\subseteq N_G(v)$. Otherwise, there exists some vertex in $U_2$ not adjacent to $v$, say $v\not\sim y_1$. Now, if $v\not\sim x_1$ (see Fig. \ref{fig-3} (3)), then we get the induced subgraph $G[v,v_1,x_1,y_1,v_4]=P_5$, a contradiction; if $v\sim x_1$ (see Fig. \ref{fig-3} (4)), then we get the induced subgraph $G[v_1,x_1,y_1,v_4,v]=I_2$, a contradiction.

Next we will show that  $N_G(v)\cap U_1=\emptyset$. Otherwise, there exists some vertex in $U_1$  adjacent to $v$, say  $v\sim x_1$. Recall that $v\sim y_1$ (see Fig. \ref{fig-3} (5)) according to the above arguments, we get the induced subgraph $G[v_1,x_1,y_1,v_4,v]=I_5$, a contradiction.

Summariszing the above discussion, we know that $V(G')\cup\{v\}$ induces a subgraph $J(a+1,b)$ of $G$. This is impossible since $G'=J(a,b)$ is assumed to be the maximal induced subgraph including $J(2,1)$.

{\flushleft\bf Case 2.} $v\sim v_4$ and $v\not\sim v_1$;

As similar as Case 1, by symmetry we can also deduce that  $G[V(G')\cup \{v\}]=J(a,b+1)$. This is also impossible.

We complete this proof.
\end{proof}
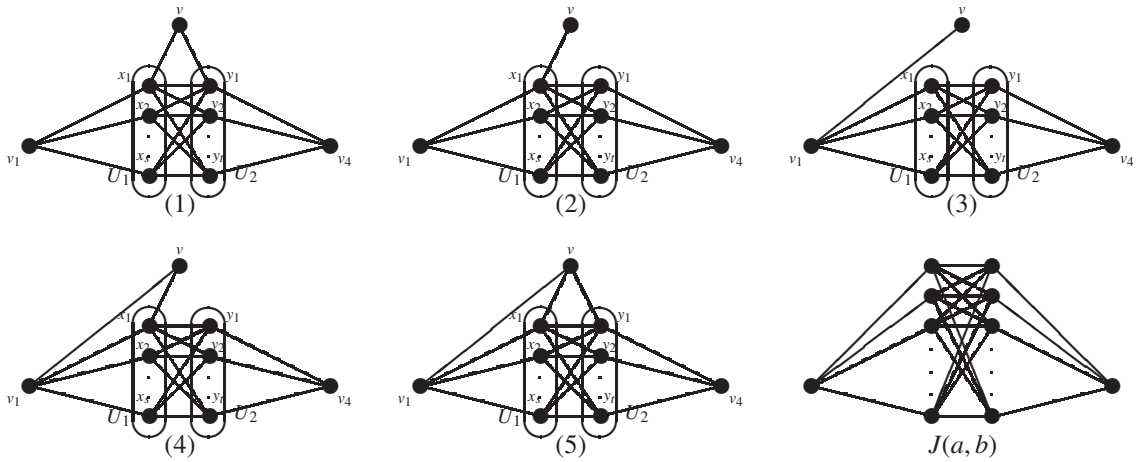
\begin{figure}[H]
\unitlength 4mm 
\linethickness{0.4pt}
\ifx\plotpoint\undefined\newsavebox{\plotpoint}\fi 
\begin{picture}(38,15)(0,0)
\thicklines
\multiput(1,10)(.06666667,.03333333){60}{\line(1,0){.06666667}}
\multiput(1,10)(.1333333,.0333333){30}{\line(1,0){.1333333}}
\multiput(1,10)(.1333333,-.0333333){30}{\line(1,0){.1333333}}
\multiput(7,12)(.06666667,-.03333333){60}{\line(1,0){.06666667}}
\multiput(11,10)(-.1333333,.0333333){30}{\line(-1,0){.1333333}}
\multiput(11,10)(-.1333333,-.0333333){30}{\line(-1,0){.1333333}}
\put(5,12){\line(1,0){2}}
\multiput(5,12)(.0666667,-.0333333){30}{\line(1,0){.0666667}}
\multiput(5,12)(.03333333,-.05){60}{\line(0,-1){.05}}
\multiput(5,11)(.0666667,.0333333){30}{\line(1,0){.0666667}}
\put(5,11){\line(1,0){2}}
\multiput(7,11)(-.03333333,-.03333333){60}{\line(0,-1){.03333333}}
\multiput(5,11)(.03333333,-.03333333){60}{\line(0,-1){.03333333}}
\multiput(7,12)(-.03333333,-.05){60}{\line(0,-1){.05}}
\put(5,9){\line(1,0){2}}
\put(0.5,9.5){\makebox(0,0)[cc]{\tiny$v_1$}}
\put(4.2,12.3){\makebox(0,0)[cc]{\tiny$x_1$}}
\put(7.8,12.3){\makebox(0,0)[cc]{\tiny$y_1$}}
\put(11.5,9.5){\makebox(0,0)[cc]{\tiny$v_4$}}
\put(4.8,11.3){\makebox(0,0)[cc]{\tiny$x_2$}}
\put(4.8,9.6){\makebox(0,0)[cc]{\tiny$x_s$}}
\put(7.3,11.3){\makebox(0,0)[cc]{\tiny$y_2$}}
\put(7.3,9.6){\makebox(0,0)[cc]{\tiny$y_t$}}
\multiput(5,12)(.0333333,.0666667){30}{\line(0,1){.0666667}}
\multiput(6,14)(.0333333,-.0666667){30}{\line(0,-1){.0666667}}
\put(6,14.5){\makebox(0,0)[cc]{\tiny$v$}}
\multiput(4.93,10.93)(0,-.6667){4}{{\rule{.8pt}{.8pt}}}
\multiput(6.93,10.93)(0,-.6667){4}{{\rule{.8pt}{.8pt}}}
\put(1,10){\circle*{.5}}
\put(5,12){\circle*{.5}}
\put(7,12){\circle*{.5}}
\put(6,14){\circle*{.5}}
\put(5,11){\circle*{.5}}
\put(7,11){\circle*{.5}}
\put(11,10){\circle*{.5}}
\put(7,9){\circle*{.5}}
\put(5,9){\circle*{.5}}
\put(4.988,10.478){\oval(1.078,4.348)[]}
\put(6.94,10.46){\oval(1.115,4.311)[]}
\put(3.966,8.973){\makebox(0,0)[cc]{\scriptsize$U_1$}}
\put(8.203,9.085){\makebox(0,0)[cc]{\scriptsize$U_2$}}
\put(6,8){\makebox(0,0)[cc]{\footnotesize(1)}}
\multiput(14,10)(.06666667,.03333333){60}{\line(1,0){.06666667}}
\multiput(14,10)(.1333333,.0333333){30}{\line(1,0){.1333333}}
\multiput(14,10)(.1333333,-.0333333){30}{\line(1,0){.1333333}}
\multiput(20,12)(.06666667,-.03333333){60}{\line(1,0){.06666667}}
\multiput(24,10)(-.1333333,.0333333){30}{\line(-1,0){.1333333}}
\multiput(24,10)(-.1333333,-.0333333){30}{\line(-1,0){.1333333}}
\put(18,12){\line(1,0){2}}
\multiput(18,12)(.0666667,-.0333333){30}{\line(1,0){.0666667}}
\multiput(18,12)(.03333333,-.05){60}{\line(0,-1){.05}}
\multiput(18,11)(.0666667,.0333333){30}{\line(1,0){.0666667}}
\put(18,11){\line(1,0){2}}
\multiput(20,11)(-.03333333,-.03333333){60}{\line(0,-1){.03333333}}
\multiput(18,11)(.03333333,-.03333333){60}{\line(0,-1){.03333333}}
\multiput(20,12)(-.03333333,-.05){60}{\line(0,-1){.05}}
\put(18,9){\line(1,0){2}}
\put(13.5,9.5){\makebox(0,0)[cc]{\tiny$v_1$}}
\put(17.2,12.3){\makebox(0,0)[cc]{\tiny$x_1$}}
\put(20.8,12.3){\makebox(0,0)[cc]{\tiny$y_1$}}
\put(24.5,9.5){\makebox(0,0)[cc]{\tiny$v_4$}}
\put(17.8,11.3){\makebox(0,0)[cc]{\tiny$x_2$}}
\put(17.8,9.6){\makebox(0,0)[cc]{\tiny$x_s$}}
\put(20.3,11.3){\makebox(0,0)[cc]{\tiny$y_2$}}
\put(20.3,9.6){\makebox(0,0)[cc]{\tiny$y_t$}}
\multiput(18,12)(.0333333,.0666667){30}{\line(0,1){.0666667}}
\put(19,14.5){\makebox(0,0)[cc]{\tiny$v$}}
\multiput(17.93,10.93)(0,-.6667){4}{{\rule{.8pt}{.8pt}}}
\multiput(19.93,10.93)(0,-.6667){4}{{\rule{.8pt}{.8pt}}}
\put(14,10){\circle*{.5}}
\put(18,12){\circle*{.5}}
\put(20,12){\circle*{.5}}
\put(19,14){\circle*{.5}}
\put(18,11){\circle*{.5}}
\put(20,11){\circle*{.5}}
\put(24,10){\circle*{.5}}
\put(20,9){\circle*{.5}}
\put(18,9){\circle*{.5}}
\put(17.988,10.478){\oval(1.078,4.348)[]}
\put(19.94,10.46){\oval(1.115,4.311)[]}
\put(16.966,8.973){\makebox(0,0)[cc]{\scriptsize$U_1$}}
\put(21.203,9.085){\makebox(0,0)[cc]{\scriptsize$U_2$}}
\put(19,8){\makebox(0,0)[cc]{\footnotesize(2)}}
\multiput(27,10)(.06666667,.03333333){60}{\line(1,0){.06666667}}
\multiput(27,10)(.1333333,.0333333){30}{\line(1,0){.1333333}}
\multiput(27,10)(.1333333,-.0333333){30}{\line(1,0){.1333333}}
\multiput(33,12)(.06666667,-.03333333){60}{\line(1,0){.06666667}}
\multiput(37,10)(-.1333333,.0333333){30}{\line(-1,0){.1333333}}
\multiput(37,10)(-.1333333,-.0333333){30}{\line(-1,0){.1333333}}
\put(31,12){\line(1,0){2}}
\multiput(31,12)(.0666667,-.0333333){30}{\line(1,0){.0666667}}
\multiput(31,12)(.03333333,-.05){60}{\line(0,-1){.05}}
\multiput(31,11)(.0666667,.0333333){30}{\line(1,0){.0666667}}
\put(31,11){\line(1,0){2}}
\multiput(33,11)(-.03333333,-.03333333){60}{\line(0,-1){.03333333}}
\multiput(31,11)(.03333333,-.03333333){60}{\line(0,-1){.03333333}}
\multiput(33,12)(-.03333333,-.05){60}{\line(0,-1){.05}}
\put(31,9){\line(1,0){2}}
\put(26.5,9.5){\makebox(0,0)[cc]{\tiny$v_1$}}
\put(30.2,12.3){\makebox(0,0)[cc]{\tiny$x_1$}}
\put(33.8,12.3){\makebox(0,0)[cc]{\tiny$y_1$}}
\put(37.5,9.5){\makebox(0,0)[cc]{\tiny$v_4$}}
\put(30.8,11.3){\makebox(0,0)[cc]{\tiny$x_2$}}
\put(30.8,9.6){\makebox(0,0)[cc]{\tiny$x_s$}}
\put(33.3,11.3){\makebox(0,0)[cc]{\tiny$y_2$}}
\put(33.3,9.6){\makebox(0,0)[cc]{\tiny$y_t$}}
\put(32,14.5){\makebox(0,0)[cc]{\tiny$v$}}
\multiput(30.93,10.93)(0,-.6667){4}{{\rule{.8pt}{.8pt}}}
\multiput(32.93,10.93)(0,-.6667){4}{{\rule{.8pt}{.8pt}}}
\put(27,10){\circle*{.5}}
\put(31,12){\circle*{.5}}
\put(33,12){\circle*{.5}}
\put(32,14){\circle*{.5}}
\put(31,11){\circle*{.5}}
\put(33,11){\circle*{.5}}
\put(37,10){\circle*{.5}}
\put(33,9){\circle*{.5}}
\put(31,9){\circle*{.5}}
\put(30.988,10.478){\oval(1.078,4.348)[]}
\put(32.94,10.46){\oval(1.115,4.311)[]}
\put(29.966,8.973){\makebox(0,0)[cc]{\scriptsize$U_1$}}
\put(34.203,9.085){\makebox(0,0)[cc]{\scriptsize$U_2$}}
\put(32,8){\makebox(0,0)[cc]{\footnotesize(3)}}
\multiput(1,2)(.06666667,.03333333){60}{\line(1,0){.06666667}}
\multiput(1,2)(.1333333,.0333333){30}{\line(1,0){.1333333}}
\multiput(1,2)(.1333333,-.0333333){30}{\line(1,0){.1333333}}
\multiput(7,4)(.06666667,-.03333333){60}{\line(1,0){.06666667}}
\multiput(11,2)(-.1333333,.0333333){30}{\line(-1,0){.1333333}}
\multiput(11,2)(-.1333333,-.0333333){30}{\line(-1,0){.1333333}}
\put(5,4){\line(1,0){2}}
\multiput(5,4)(.0666667,-.0333333){30}{\line(1,0){.0666667}}
\multiput(5,4)(.03333333,-.05){60}{\line(0,-1){.05}}
\multiput(5,3)(.0666667,.0333333){30}{\line(1,0){.0666667}}
\put(5,3){\line(1,0){2}}
\multiput(7,3)(-.03333333,-.03333333){60}{\line(0,-1){.03333333}}
\multiput(5,3)(.03333333,-.03333333){60}{\line(0,-1){.03333333}}
\multiput(7,4)(-.03333333,-.05){60}{\line(0,-1){.05}}
\put(5,1){\line(1,0){2}}
\put(0.5,1.5){\makebox(0,0)[cc]{\tiny$v_1$}}
\put(4.2,4.3){\makebox(0,0)[cc]{\tiny$x_1$}}
\put(7.8,4.3){\makebox(0,0)[cc]{\tiny$y_1$}}
\put(11.5,1.5){\makebox(0,0)[cc]{\tiny$v_4$}}
\put(4.8,3.3){\makebox(0,0)[cc]{\tiny$x_2$}}
\put(4.8,1.6){\makebox(0,0)[cc]{\tiny$x_s$}}
\put(7.3,3.3){\makebox(0,0)[cc]{\tiny$y_2$}}
\put(7.3,1.6){\makebox(0,0)[cc]{\tiny$y_t$}}
\multiput(5,4)(.0333333,.0666667){30}{\line(0,1){.0666667}}
\put(6,6.5){\makebox(0,0)[cc]{\tiny$v$}}
\multiput(4.93,2.93)(0,-.6667){4}{{\rule{.8pt}{.8pt}}}
\multiput(6.93,2.93)(0,-.6667){4}{{\rule{.8pt}{.8pt}}}
\put(1,2){\circle*{.5}}
\put(5,4){\circle*{.5}}
\put(7,4){\circle*{.5}}
\put(6,6){\circle*{.5}}
\put(5,3){\circle*{.5}}
\put(7,3){\circle*{.5}}
\put(11,2){\circle*{.5}}
\put(7,1){\circle*{.5}}
\put(5,1){\circle*{.5}}
\put(4.988,2.478){\oval(1.078,4.348)[]}
\put(6.94,2.46){\oval(1.115,4.311)[]}
\put(3.966,.973){\makebox(0,0)[cc]{\scriptsize$U_1$}}
\put(8.203,1.085){\makebox(0,0)[cc]{\scriptsize$U_2$}}
\put(6,0){\makebox(0,0)[cc]{\footnotesize(4)}}
\multiput(14,2)(.06666667,.03333333){60}{\line(1,0){.06666667}}
\multiput(14,2)(.1333333,.0333333){30}{\line(1,0){.1333333}}
\multiput(14,2)(.1333333,-.0333333){30}{\line(1,0){.1333333}}
\multiput(20,4)(.06666667,-.03333333){60}{\line(1,0){.06666667}}
\multiput(24,2)(-.1333333,.0333333){30}{\line(-1,0){.1333333}}
\multiput(24,2)(-.1333333,-.0333333){30}{\line(-1,0){.1333333}}
\put(18,4){\line(1,0){2}}
\multiput(18,4)(.0666667,-.0333333){30}{\line(1,0){.0666667}}
\multiput(18,4)(.03333333,-.05){60}{\line(0,-1){.05}}
\multiput(18,3)(.0666667,.0333333){30}{\line(1,0){.0666667}}
\put(18,3){\line(1,0){2}}
\multiput(20,3)(-.03333333,-.03333333){60}{\line(0,-1){.03333333}}
\multiput(18,3)(.03333333,-.03333333){60}{\line(0,-1){.03333333}}
\multiput(20,4)(-.03333333,-.05){60}{\line(0,-1){.05}}
\put(18,1){\line(1,0){2}}
\put(13.5,1.5){\makebox(0,0)[cc]{\tiny$v_1$}}
\put(17.2,4.3){\makebox(0,0)[cc]{\tiny$x_1$}}
\put(20.8,4.3){\makebox(0,0)[cc]{\tiny$y_1$}}
\put(24.5,1.5){\makebox(0,0)[cc]{\tiny$v_4$}}
\put(17.8,3.3){\makebox(0,0)[cc]{\tiny$x_2$}}
\put(17.8,1.6){\makebox(0,0)[cc]{\tiny$x_s$}}
\put(20.3,3.3){\makebox(0,0)[cc]{\tiny$y_2$}}
\put(20.3,1.6){\makebox(0,0)[cc]{\tiny$y_t$}}
\multiput(18,4)(.0333333,.0666667){30}{\line(0,1){.0666667}}
\multiput(19,6)(.0333333,-.0666667){30}{\line(0,-1){.0666667}}
\put(19,6.5){\makebox(0,0)[cc]{\tiny$v$}}
\multiput(17.93,2.93)(0,-.6667){4}{{\rule{.8pt}{.8pt}}}
\multiput(19.93,2.93)(0,-.6667){4}{{\rule{.8pt}{.8pt}}}
\put(14,2){\circle*{.5}}
\put(18,4){\circle*{.5}}
\put(20,4){\circle*{.5}}
\put(19,6){\circle*{.5}}
\put(18,3){\circle*{.5}}
\put(20,3){\circle*{.5}}
\put(24,2){\circle*{.5}}
\put(20,1){\circle*{.5}}
\put(18,1){\circle*{.5}}
\put(17.988,2.478){\oval(1.078,4.348)[]}
\put(19.94,2.46){\oval(1.115,4.311)[]}
\put(16.966,.973){\makebox(0,0)[cc]{\scriptsize$U_1$}}
\put(21.203,1.085){\makebox(0,0)[cc]{\scriptsize$U_2$}}
\put(19,0){\makebox(0,0)[cc]{\footnotesize(5)}}
\put(27,2){\line(1,1){4}}
\put(27,2){\line(4,3){4}}
\multiput(27,2)(.06666667,.03333333){60}{\line(1,0){.06666667}}
\multiput(27,2)(.1333333,-.0333333){30}{\line(1,0){.1333333}}
\put(31,1){\line(1,0){2}}
\multiput(33,1)(.1333333,.0333333){30}{\line(1,0){.1333333}}
\multiput(37,2)(-.06666667,.03333333){60}{\line(-1,0){.06666667}}
\put(37,2){\line(-4,3){4}}
\put(37,2){\line(-1,1){4}}
\put(31,6){\line(1,0){2}}
\multiput(33,6)(-.0666667,-.0333333){30}{\line(-1,0){.0666667}}
\multiput(33,6)(-.03333333,-.03333333){60}{\line(0,-1){.03333333}}
\put(33,6){\line(-2,-5){2}}
\multiput(31,6)(.0666667,-.0333333){30}{\line(1,0){.0666667}}
\multiput(31,6)(.03333333,-.03333333){60}{\line(0,-1){.03333333}}
\put(31,6){\line(2,-5){2}}
\put(31,5){\line(1,0){2}}
\multiput(31,5)(.0666667,-.0333333){30}{\line(1,0){.0666667}}
\multiput(31,5)(.03333333,-.06666667){60}{\line(0,-1){.06666667}}
\put(31,4){\line(1,0){2}}
\multiput(33,4)(-.03333333,-.05){60}{\line(0,-1){.05}}
\multiput(31,4)(.03333333,-.05){60}{\line(0,-1){.05}}
\multiput(33,5)(-.0666667,-.0333333){30}{\line(-1,0){.0666667}}
\multiput(33,5)(-.03333333,-.06666667){60}{\line(0,-1){.06666667}}
\put(27,2){\circle*{.5}}
\put(31,6){\circle*{.5}}
\put(31,5){\circle*{.5}}
\multiput(30.93,3.93)(0,-.75){5}{{\rule{.8pt}{.8pt}}}
\multiput(32.93,3.93)(0,-.75){5}{{\rule{.8pt}{.8pt}}}
\put(31,4){\circle*{.5}}
\put(31,1){\circle*{.5}}
\put(33,1){\circle*{.5}}
\put(33,4){\circle*{.5}}
\put(33,5){\circle*{.5}}
\put(33,6){\circle*{.5}}
\put(32,0){\makebox(0,0)[cc]{\footnotesize$J(a,b)$}}
\put(27,10){\line(5,4){5}}
\put(1,2){\line(5,4){5}}
\put(19,6){\line(-5,-4){5}}
\put(37,2){\circle*{.5}}
\end{picture}
  \caption{The graphs used in Lemma \ref{lem-x-3-6}}\label{fig-3}
\end{figure}

After the completion of the preparations, we get one of our main results.
\begin{thm}\label{thm-3-1}
Let $G\in\mathcal{G}(n,n-3)$ with $n\ge 6$, then $d(G)=2$.
\end{thm}
\begin{proof}
By Lemma \ref{lem-3-1} and Remark \ref{re-1}, we get that $G$ is $P_5$-free and $d(G)=2$ or $d(G)=3$. Assume by contradiction that $d(G)=3$. Let $d(v_1,v_4)=3$ and $P=v_1v_2v_3v_4$ a shortest path between $v_1$ and $v_4$. By Lemma \ref{lem-3-4}, $G$ contains at least one of $I_i$ (labelled as Fig. \ref{fig-1}) as an induced subgraph for $i=1,2,\ldots,5$.

Suppose that $I_1$ is an induced subgraph of $G$. Note that $d_G(v_4,u)=d_{I_1}(v_4,u)-1=2$ or $d_G(v_4,u)=d_{I_1}(v_4,u)=3$. We get that either $M_1$ or $M_1'$ is a principal submatrix of $\mathcal{D}^L(G)$ with respect to $I_1$, where
\[M_1=\left(\begin{array}{ccccc}t_1&-1&-2&-3&-2\\-1&t_2&-1&-2&-1\\-2&-1&t_3&-1&-2\\-3&-2&-1&t_4&-2\\-2&-1&-2&-2&t_5\end{array}\right)\begin{array}{l}v_1\\v_2\\v_3\\v_4\\u\end{array},M_1'=\left(\begin{array}{ccccc}t_1&-1&-2&-3&-2\\-1&t_2&-1&-2&-1\\-2&-1&t_3&-1&-2\\-3&-2&-1&t_4&-3\\-2&-1&-2&-3&t_5\end{array}\right)\begin{array}{l}v_1\\v_2\\v_3\\v_4\\u\end{array}.\]
If $M_1$ is a principal submatrix of $\mathcal{D}^L(G)$, by Lemma \ref{lem-3-5}, there exists an eigenvector $x=(x_1,x_2,x_3,x_4,0)$ satisfying $x_1+x_2+x_3+x_4=0$ such that $M_1x=\partial^L_1x$. Consider the fifth entry of both sides of $M_1x=\partial^L_1x$, we have $-2x_1-x_2-2x_3-2x_4=0$. It follows that $x_2=0$ and $x_1+x_3+x_4=0$. Next we consider the second entry of both sides of $M_1x=\partial^L_1x$, we have $-x_1-x_3-2x_4=0$. It follows that $x_4=0$ and $x_1+x_3=0$. We consider the fourth entry of both sides of $M_1x=\partial^L_1x$, we have $-3x_1-x_3=0$. It follows that $x_1=x_3=0$. Thus, we have $x=0$, a contradiction.
If $M_1'$ is a principal submatrix of $\mathcal{D}^L(G)$, by Lemma \ref{lem-3-5}, there exists an eigenvector $y=(0,y_2,y_3,y_4,y_5)^T$ satisfying
\begin{equation}\label{eq-1}y_2+y_3+y_4+y_5=0\end{equation} such that
\begin{equation}\label{eq-2}M_1'y=\partial^L_1y.\end{equation}
Consider the first entry of both sides of Eq. (\ref{eq-2}), we have \begin{equation}\label{eq-3}-y_2-2y_3-3y_4-2y_5=0.\end{equation}
Combining (\ref{eq-1}) and (\ref{eq-3}), we have $y_2=y_4$. If $y_2=y_4=0$, we consider the fourth entry of both sides of (\ref{eq-2}) and we get that $y_5=0$. It follows that $y=0$, a contradiction. If $y_2=y_4\ne 0$, we consider the second entry of both sides of (\ref{eq-2}) and we get that $\partial^L_1=t_2-\frac{y_3+2y_4+y_5}{y_2}=t_2-\frac{y_2+y_3+y_4+y_5}{y_2}$. From (\ref{eq-1}), we have $\partial^L_1=t_2$. It contradicts Corollary \ref{cor-3-1}.

Suppose that $I_2$ is an induced subgraph of $G$. Note that $d_G(v_4,u)=d_{I_2}(v_4,u)=3$ or $d_G(v_4,u)=d_{I_2}(v_4,u)-1=2$. We get that the matrix $M_2$ or $M_2'$ is a principal submatrix of $\mathcal{D}^L(G)$ with respect to $I_2$, where
\[M_2=\left(\begin{array}{ccccc}t_1&-1&-2&-3&-1\\-1&t_2&-1&-2&-1\\-2&-1&t_3&-1&-2\\-3&-2&-1&t_4&-3\\-1&-1&-2&-3&t_5\end{array}\right)\begin{array}{l}v_1\\v_2\\v_3\\v_4\\u\end{array}, M_2'=\left(\begin{array}{ccccc}t_1&-1&-2&-3&-1\\-1&t_2&-1&-2&-1\\-2&-1&t_3&-1&-2\\-3&-2&-1&t_4&-2\\-1&-1&-2&-2&t_5\end{array}\right)\begin{array}{l}v_1\\v_2\\v_3\\v_4\\u\end{array}.\]
If $M_2$ is a principal submatrix of $\mathcal{D}^L(G)$, by Lemma \ref{lem-3-5}, there exists an eigenvector $x=(x_1,0,x_3,x_4,x_5)^T$ satisfying $x_1+x_3+x_4+x_5=0$ such that $M_2x=\partial^L_1x$. We successively consider the second, the fourth and the third entries of both sides of $M_2x=\partial^L_1x$, we get that $x_3=x_4=0$ and $x_1+x_5=0$. If $x_1=0$, then $x_5=0$ and $x=0$, a contradiction. If $x_1\ne0$, consider the first entry of both sides of $M_2x=\partial^L_1x$, we get that $\partial^L_1=t_1-\frac{x_5}{x_1}=t_1+1$. It contradicts Corollary \ref{cor-3-1}.
If $M_2'$ is a principal submatrix of $\mathcal{D}^L(G)$, by Lemma \ref{lem-3-5}, there exists an eigenvector $y=(y_1,y_2,y_3,y_4,0)^T$ satisfying $y_1+y_2+y_3+y_4=0$ such that $M_2'y=\partial^L_1y$. Consider the fifth entry of both sides of $M_2'y=\partial^L_1y$, we have
$-y_1-y_2-2y_3-2y_4=0$. It leads to that $y_1+y_2=y_3+y_4=0$. If $y_3=y_4=0$, we consider the third entry of both sides of $M_2'y=\partial^L_1y$ and we get that $y_1=y_2=0$. It leads to that $y=0$, a contradiction. If $y_1=y_2=0$, we consider the second entry of both sides of $M_2'y=\partial^L_1y$ and we get that $y_3=y_4=0$. It leads to that $y=0$, a contradiction. If $y_1,y_2,y_3,y_4\ne 0$, without loss of generality, we may suppose that $y=(a,-a,1,-1,0)$. Consider the third entry of both sides of $M_2'y=\partial^L_1y$, we have $\partial^L_1=t_3+1-a$. By Corollary \ref{cor-3-1}, we have $a<0$. Consider the fourth entry of both sides of $M_2'y=\partial^L_1y$, we have $\partial^L_1=t_4+1+a$. By Corollary \ref{cor-3-1}, we have $a>0$, a contradiction.

Suppose that $I_4$ is an induced subgraph of $G$. We get that the matrix $M_4$ is a principal submatrix of $\mathcal{D}^L(G)$ with respect to $I_4$, where
\[M_4=\left(\begin{array}{ccccc}t_1&-1&-2&-3&-2\\-1&t_2&-1&-2&-1\\-2&-1&t_3&-1&-1\\-3&-2&-1&t_4&-2\\-2&-1&-1&-2&t_5\end{array}\right)\begin{array}{l}v_1\\v_2\\v_3\\v_4\\u\end{array}.\]
By Lemma \ref{lem-3-5}, there exists an eigenvector $x=(x_1,0,x_3,x_4,x_5)^T$ satisfying $x_1+x_3+x_4+x_5=0$ such that $M_4x=\partial^L_1x$. Consider the second and the fourth entries of both sides of $M_4x=\partial^L_1x$ successively, we get that $x_4=0$, $x_1=x_3$ and $x_5=-2x_1$. If $x_1=x_3=0$, then $x=0$, a contradiction. If $x_1=x_3\ne 0$, consider the third entry of both sides of $M_4x=\partial^L_1x$ and we get that $\partial^L_1=t_3$. It contradicts Corollary \ref{cor-3-1}.

Suppose that $I_5$ is an induced subgraph of $G$. We get that the matrix $M_5$ is a principal submatrix of $\mathcal{D}^L(G)$ with respect to $I_5$, where
\[M_5=\left(\begin{array}{ccccc}t_1&-1&-2&-3&-1\\-1&t_2&-1&-2&-1\\-2&-1&t_3&-1&-1\\-3&-2&-1&t_4&-2\\-1&-1&-1&-2&t_5\end{array}\right)\begin{array}{l}v_1\\v_2\\v_3\\v_4\\u\end{array}.\]
By Lemma \ref{lem-3-5}, there exists an eigenvector $x=(x_1,0,x_3,x_4,x_5)^T$ satisfying $x_1+x_3+x_4+x_5=0$ such that $M_5x=\partial^L_1x$. We successively consider the second and the fourth entries of both sides of $M_5x=\partial^L_1x$, we have that $x_4=0$, $x_1=x_3$ and $x_5=-2x_3$. If $x_3=0$, we have $x=0$, a contradiction. If $x_3\ne0$, consider the third entry of both sides of $M_5x=\partial^L_1x$, we have $\partial^L_1=t_3$. It contradicts Corollary \ref{cor-3-1}.

Suppose that $I_3$ is an induced subgraph of $G$. On the one hand, we get that the matrix $M_3$ is a principal submatrix of $\mathcal{D}^L(G)$ with respect to $I_3$, where
\[M_3=\left(\begin{array}{ccccc}t_1&-1&-2&-3&-1\\-1&t_2&-1&-2&-2\\-2&-1&t_3&-1&-1\\-3&-2&-1&t_4&-2\\-1&-2&-1&-2&t_5\end{array}\right)\begin{array}{l}v_1\\v_2\\v_3\\v_4\\u\end{array}.\]
By Lemma \ref{lem-3-5}, there exists an eigenvector $x=(x_1,x_2,0,x_4,x_5)^T$ satisfying $x_1+x_2+x_4+x_5=0$ such that $M_3x=\partial^L_1x$. We successively consider the third, the first and the fourth entries of both sides of $M_3x=\partial^L_1x$, then we get that $x_1=x_4=0$ and $x_2+x_5=0$. If $x_2=x_5=0$, then $x=0$, a contradiction. If $x_2\ne 0$, without loss of generality, we may suppose that $x=(0,1,0,0,-1)^T$. Consider the second entry of $M_3x=\partial^L_1x$, we get that $\partial^L_1=t_2+2$.  By Corollary \ref{cor-3-1}, we get that
\begin{equation}\label{eq-4}Tr(v_2)=\max_{v\in V(G)}Tr(v).\end{equation}
On the other hand, recall that $G$ is $P_5$-free. Moreover, by the arguments above, we have that $G$ contains no induced $I_1$, $I_2$, $I_4$ or $I_5$. Therefore, by Lemma \ref{lem-x-3-6}, we have that $G=J(a,b)$ with roots $v_1$ and $v_4$. By simple calculation, we have $Tr(v_1)=a+2b+3$, $Tr(v_4)=2a+b+3$, $Tr(x)=2a+b+1$ for every $x\in N(v_1)$ and $Tr(y)=2b+a+1$ for every $y\in N(v_4)$. Note that $v_2\in N(v_1)$. We get that
\[Tr(v_2)=2a+b+1<2a+b+3=Tr(v_4),\]
which contradicts (\ref{eq-4}).

We complete the proof.
\end{proof}

The result above showed that the graphs in $\mathcal{G}(n,n-3)$ have diameter $2$. In fact, we can further obtain that $G$ is the join of two graphs. To prove this, we need the following result.
\begin{lem}\label{lem-x-3-7}
Let $G\in\mathcal{G}(n,n-3)$ with $n\ge 6$, then none of $J_1(=C_5)$, $J_2$ or $J_3$ (shown in Fig. \ref{fig-2}) can be an induced subgraph of $G$.
\end{lem}
\begin{proof}
Assume by contradiction that $J_1=C_5$ is an induced subgraph of $G$. We get that the matrix $N_1$ is a principal submatrix of $\mathcal{D}^L(G)$ with respect to $J_1$, where
\[N_1=\left(\begin{array}{ccccc}t_1&-1&-2&-2&-1\\-1&t_2&-1&-2&-2\\-2&-1&t_3&-1&-2\\-2&-2&-1&t_4&-1\\-1&-2&-2&-1&t_5\end{array}\right)\begin{array}{l}v_1\\v_2\\v_3\\v_4\\u\end{array}.\]
By Lemma \ref{lem-3-5}, there exists an eigenvector $x=(0,x_2,x_3,x_4,x_5)^T$ satisfying $x_2+x_3+x_4+x_5=0$ such that $N_1x=\partial^L_1x$. From the first  entry of $N_1x=\partial^L_1x$, we have $-x_2-2x_3-2x_4-x_5=0$. Therefore, we have $x_3+x_4=0$ and $x_2+x_5=0$. If $x_3=x_4=0$, consider the third entry of both sides of $N_1x=\partial^L_1x$ and we get that $x_2=x_5=0$. It leads to that $x=0$, a contradiction. If $x_2=x_5=0$, consider the second entry of both sides of $N_1x=\partial^L_1x$ and we get that $x_3=x_4=0$. It leads to that $x=0$, a contradiction. If $x_2, x_3,x_4,x_5\ne 0$, without loss of generality, we may suppose that $x=(0,a,1,-1,-a)^T$. Thus, we have
\begin{equation}\label{eq-5}
\left(\begin{array}{ccccc}t_1&-1&-2&-2&-1\\-1&t_2&-1&-2&-2\\-2&-1&t_3&-1&-2\\-2&-2&-1&t_4&-1\\-1&-2&-2&-1&t_5\end{array}\right)\left(\begin{array}{c}0\\a\\1\\-1\\-a\end{array}\right)=\partial^L_1\left(\begin{array}{c}0\\a\\1\\-1\\-a\end{array}\right).
\end{equation}
Consider the fourth entry of both sides of (\ref{eq-5}), we have
\[\partial^L_1=t_4+a+1.\]
By Corollary \ref{cor-3-1}, we have $a\ge1$. Consider the fifth entry of both sides of Eq. (\ref{eq-5}), we have
\[\partial^L_1=t_5+\frac{1}{a}+2.\]
By Lemma \ref{lem-3-3}, we get that $\partial^L_1$ is integral. Therefore, $a$ and $\frac{1}{a}$ are both integral. Thus, we have $a=1$ and $\partial^L_1=t_4+2=t_5+3$. It follows that
\begin{equation}\label{eq-6}t_4=t_5+1.\end{equation}
On the other hand, by Lemma\ref{lem-3-5}, there also exists an eigenvector $y=(y_1,0,y_3,y_4,y_5)^T$ satisfying $y_1+y_3+y_4+y_5=0$ such that $N_1y=\partial^L_1y$. From the second entry of $N_1y=\partial^L_1y$, we have $-y_1-y_3-2y_4-2y_5=0$. Therefore, we have $y_4+y_5=0$ and $y_1+y_3=0$. If $y_1=y_3=0$ or $y_4=y_5=0$, we also get $y=0$, a contradiction. If $y_1,y_3,y_4,y_5\ne 0$, without loss of generality, we may suppose that $y=(b,0,-b,1,-1)^T$. Thus, we have
 \begin{equation}\label{eq-7}
\left(\begin{array}{ccccc}t_1&-1&-2&-2&-1\\-1&t_2&-1&-2&-2\\-2&-1&t_3&-1&-2\\-2&-2&-1&t_4&-1\\-1&-2&-2&-1&t_5\end{array}\right)\left(\begin{array}{c}b\\0\\-b\\1\\-1\end{array}\right)=\partial^L_1\left(\begin{array}{c}b\\0\\-b\\1\\-1\end{array}\right).
\end{equation}
 Consider the fourth and the fifth entries of both sides of Eq. (\ref{eq-7}), we have
\[\partial^L_1=t_4-b+1=t_5-b+1.\]
It follows that $t_4=t_5$, which contradicts (\ref{eq-6}).

Assume by contradiction that $J_2$ is an induced subgraph of $G$. We get that the matrix $N_2$ is a principal submatrix of $\mathcal{D}^L$ with respect to $J_2$, where
\[N_2=\left(\begin{array}{ccccc}t_1&-1&-2&-2&-1\\-1&t_2&-1&-2&-1\\-2&-1&t_3&-1&-2\\-2&-2&-1&t_4&-1\\-1&-1&-2&-1&t_5\end{array}\right)\begin{array}{l}v_1\\v_2\\v_3\\v_4\\u\end{array}.\]
By Lemma \ref{lem-3-5}, there exists an eigenvector $x=(x_1,x_2,x_3,x_4,0)^T$ satisfying $x_1+x_2+x_3+x_4=0$ such that $N_2x=\partial^L_1x$. We successively consider the fifth, the third, the first and the fourth entries of both sides of $N_2x=\partial^L_1x$, then we get that $x=0$, a contradiction.

Assume by contradiction that $J_3$ is an induced subgraph of $G$. We get that the matrix $N_3$ is a principal submatrix of $\mathcal{D}^L$ with respect to $J_3$, where
\[N_3=\left(\begin{array}{ccccc}t_1&-1&-2&-2&-1\\-1&t_2&-1&-2&-1\\-2&-1&t_3&-1&-1\\-2&-2&-1&t_4&-1\\-1&-1&-1&-1&t_5\end{array}\right)\begin{array}{l}v_1\\v_2\\v_3\\v_4\\u\end{array}.\]
By Lemma \ref{lem-3-5}, there exists an eigenvector $x=(x_1,x_2,0,x_4,x_5)^T$ satisfying $x_1+x_2+x_4+x_5=0$ such that $N_3x=\partial^L_1x$. Consider the third, the first, the fourth and the second entries of both sides of $N_3x=\partial^L_1x$ succesively, we get that $x=0$, a contradiction.
\end{proof}
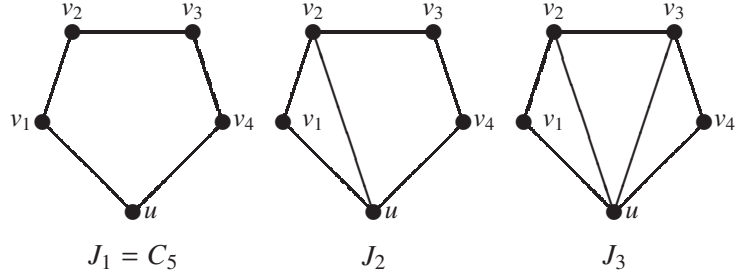
\begin{figure}[htbp]
\begin{center}
\unitlength 4mm 
\linethickness{0.4pt}
\ifx\plotpoint\undefined\newsavebox{\plotpoint}\fi 
\begin{picture}(24,9)(0,0)
\thicklines
\put(2,8){\line(1,0){4}}
\multiput(6,8)(.0333333,-.1){30}{\line(0,-1){.1}}
\multiput(7,5)(-.03370787,-.03370787){89}{\line(0,-1){.03370787}}
\multiput(4,2)(-.03370787,.03370787){89}{\line(0,1){.03370787}}
\multiput(1,5)(.0333333,.1){30}{\line(0,1){.1}}
\multiput(9,5)(.0333333,.1){30}{\line(0,1){.1}}
\put(10,8){\line(1,0){4}}
\multiput(14,8)(.0333333,-.1){30}{\line(0,-1){.1}}
\multiput(15,5)(-.03370787,-.03370787){89}{\line(0,-1){.03370787}}
\multiput(12,2)(-.03370787,.03370787){89}{\line(0,1){.03370787}}
\multiput(17,5)(.0333333,.1){30}{\line(0,1){.1}}
\put(18,8){\line(1,0){4}}
\multiput(22,8)(.0333333,-.1){30}{\line(0,-1){.1}}
\multiput(23,5)(-.03370787,-.03370787){89}{\line(0,-1){.03370787}}
\multiput(20,2)(-.03370787,.03370787){89}{\line(0,1){.03370787}}
\put(12,2){\line(-1,3){2}}
\put(18,8){\line(1,-3){2}}
\put(20,2){\line(1,3){2}}
\put(2,8){\circle*{.5}}
\put(6,8){\circle*{.5}}
\put(10,8){\circle*{.5}}
\put(14,8){\circle*{.5}}
\put(18,8){\circle*{.5}}
\put(22,8){\circle*{.5}}
\put(1,5){\circle*{.5}}
\put(7,5){\circle*{.5}}
\put(9,5){\circle*{.5}}
\put(15,5){\circle*{.5}}
\put(17,5){\circle*{.5}}
\put(23,5){\circle*{.5}}
\put(20,2){\circle*{.5}}
\put(12,2){\circle*{.5}}
\put(4,2){\circle*{.5}}
\put(0.3,5){\makebox(0,0)[cc]{\footnotesize$v_1$}}
\put(2,8.7){\makebox(0,0)[cc]{\footnotesize$v_2$}}
\put(6,8.7){\makebox(0,0)[cc]{\footnotesize$v_3$}}
\put(7.7,5){\makebox(0,0)[cc]{\footnotesize$v_4$}}
\put(10,5){\makebox(0,0)[cc]{\footnotesize$v_1$}}
\put(10,8.7){\makebox(0,0)[cc]{\footnotesize$v_2$}}
\put(14,8.7){\makebox(0,0)[cc]{\footnotesize$v_3$}}
\put(15.7,5){\makebox(0,0)[cc]{\footnotesize$v_4$}}
\put(18,5){\makebox(0,0)[cc]{\footnotesize$v_1$}}
\put(18,8.7){\makebox(0,0)[cc]{\footnotesize$v_2$}}
\put(22,8.7){\makebox(0,0)[cc]{\footnotesize$v_3$}}
\put(23.7,5){\makebox(0,0)[cc]{\footnotesize$v_4$}}
\put(4.6,2){\makebox(0,0)[cc]{\footnotesize$u$}}
\put(12.6,2){\makebox(0,0)[cc]{\footnotesize$u$}}
\put(20.6,2){\makebox(0,0)[cc]{\footnotesize$u$}}
\put(4,0.5){\makebox(0,0)[cc]{\small$J_1=C_5$}}
\put(12,0.5){\makebox(0,0)[cc]{\small$J_2$}}
\put(20,0.5){\makebox(0,0)[cc]{\small$J_3$}}
\end{picture}
\caption{The graphs in Lemma \ref{lem-x-3-7}}
\label{fig-2}
\end{center}
\end{figure}
Using the above tools, we get the following result.
\begin{thm}\label{thm-3-2}
Let $G\in\mathcal{G}(n,n-3)$ with $n\ge 6$, then $\bar{G}$ is disconnected. It means that $G$ is the join of some connected  graphs.
\end{thm}
\begin{proof}
By Lemma \ref{lem-2-2}, it suffices to show that $G$ contains no induced $P_4$. Assume by contradiction that $G$ contains an induced $P_4=v_1v_2v_3v_4$. By Theorem \ref{thm-3-1}, we have $d(G)=2$. Therefore, there exists a vertex $u\in V(G)$ such that  $u\sim v_1,v_4$. It follows that at least one of $J_1$, $J_2$ and $J_3$ will be an induced subgraph of $G$, contradicts Lemma \ref{lem-x-3-7}.
\end{proof}

For any graph $G\in\mathcal{G}(n,n-3)$, we see that $G$ has at most four distinct eigenvalues, and we also have $\partial^L_{n-1}(G)=n$ by Theorems \ref{thm-2-3} and \ref{thm-3-2}.
Denote by \[\mathcal{H}_1(n)=\{G\in\mathcal{G}(n,n-3)\mid \emph{Spec}_{\mathcal{L}}(G)=[(\partial^L_1)^{n-3},\partial^L_{n-2},\partial^L_{n-1}=n,\partial^L_{n}=0]\},\] and \[\mathcal{H}_2(n)=\{G\in\mathcal{G}(n,n-3)\mid \emph{Spec}_{\mathcal{L}}(G)=[(\partial^L_1)^{n-3},\partial^L_{n-2}=\partial^L_{n-1}=n,\partial^L_{n}=0]\}.\] Therefore, $\mathcal{H}_1(n)$ and $\mathcal{H}_2(n)$ are the sets of graphs with four and three distinct eigenvalues in $\mathcal{G}(n,n-3)$, respectively. Thus we have the disjoint decomposition
\[\mathcal{G}(n,n-3)=\mathcal{H}_1(n)\cup\mathcal{H}_2(n). \]
Mohammadian \cite{Mohammadian} gave the following result.
\begin{lem}[\cite{Mohammadian}, Theorem 8]\label{lem-x-3-8}
Let $G$ be a graph on $n\ge 5$ vertices whose distinct Laplacian eigenvalues are $0<\alpha<\beta<\gamma$. Then the multiplicity of $\alpha$ is $n-3$ if and only if $G$ is one of the graphs $K_{2,n-2}$, $K_{n/2,n/2}+e$ or $K_{1,n-1}+e$, where $K_{n/2,n/2}+e$ and $K_{1,n-1}+e$ are the graphs obtained from $K_{n/2,n/2}$ and $K_{1,n-1}$, respectively, by adding an edge $e$ joining any two non-adjacent vertices.
\end{lem}
Note that, when $d(G)=2$, there exists a correspondence between the distance Laplacian spectrum and the Laplacian spectrum of $G$. We have the following result.
\begin{cor}\label{cor-3-2}
For an integer $n\ge 6$, we have $\mathcal{H}_1(n)=\{K_{2,n-2}, K_{n/2,n/2}+e, K_{1,n-1}+e\}$, and their distance Laplacian spectra are given by
\begin{equation}\label{111}\left\{\begin{array}{l}
\emph{Spec}_{\mathcal{L}}(K_{2,n-2})=\{(2n-2)^{n-3},n+2,n,0\}\\
\emph{Spec}_{\mathcal{L}}(K_{n/2,n/2}+e)=\{(\frac{3n}{2})^{n-3},\frac{3n}{2}-2,n,0\}\\
\emph{Spec}_{\mathcal{L}}(K_{1,n-1}+e)=\{(2n-1)^{n-3},2n-3,n,0\}
\end{array}\right.\end{equation}
\end{cor}
\begin{proof}
Let $G\in \mathcal{H}_1(n)$ and $\emph{Spec}_{\mathcal{L}}(G)=\{(\partial^L_1)^{n-3},\partial^L_{n-2},n,0\}$ where $\partial^L_1>\partial^L_{n-2}>n$. By Theorem \ref{thm-3-1}, we have $d(G)=2$. Therefore, by Theorem \ref{thm-2-2}, the Laplacian spectrum of $G$ is $\{n,2n-\partial^L_{n-2},(2n-\partial^L_1)^{n-3},0\}$. Thus, we get that $G\in \{K_{2,n-2}, K_{n/2,n/2}+e, K_{1,n-1}+e\}$ from Lemma \ref{lem-x-3-8}. Conversely, note that all of $K_{2,n-2}$, $K_{n/2,n/2}+e$ and $K_{1,n-1}+e$ are the join of two graphs, by Lemma \ref{lem-2-1} (iv) and Theorem \ref{thm-2-2}, we obtain their distance Laplacian spectra, which are shown in (\ref{111}).
Therefore, $K_{2,n-2},K_{n/2,n/2}+e,K_{1,n-1}+e\in \mathcal{H}_1(n)$, and the result follows.
\end{proof}
In what follows we characterise $\mathcal{H}_2(n)$.
\begin{lem}\label{lem-3-10}
For an integer $n\ge 6$, we have $\mathcal{H}_2(n)=\{K_2\nabla (n-2)K_1, K_1\nabla K_{\frac{n-1}{2},\frac{n-1}{2}},K_{\frac{n}{3},\frac{n}{3},\frac{n}{3}}\}$, and their distance Laplacian spectra are given by
\begin{equation}\label{222}\left\{\begin{array}{l}\emph{Spec}_{\mathcal{L}}(K_2\nabla (n-2)K_1)=\{(2n-2)^{n-3},n^2,0\}\\ \emph{Spec}_{\mathcal{L}}(K_1\nabla K_{\frac{n-1}{2},\frac{n-1}{2}})=\{((3n-1)/2)^{n-3},n^2,0\}\\ \emph{Spec}_{\mathcal{L}}(K_{\frac{n}{3},\frac{n}{3},\frac{n}{3}})=\{(4n/3)^{n-3},n^2,0\}\end{array}\right.\end{equation}
\end{lem}
\begin{proof}
Let $G\in\mathcal{H}_2(n)$ and $\emph{Spec}_{\mathcal{L}}(G)=\{(\partial^L_1)^{n-3},n^2,0\}$ where $\partial^L_1>n$. By Theorem \ref{thm-3-1}, we get that $d(G)=2$. Therefore, by Theorem \ref{thm-2-2}, the Laplacian spectrum of $G$ is $\{n^2,(2n-\partial^L_1)^{n-3},0\}$. By Lemma \ref{lem-2-1} (iii), the Laplacian spectrum of $\bar{G}$ is $\{(\partial^L_1-n)^{n-3},0^3\}$. By Lemma \ref{lem-2-1} (i), $\bar{G}$ has exactly three components, denoted by $G_1$, $G_2$ and $G_3$. Moreover, by Lemma \ref{lem-2-1} (ii), $G_1$, $G_2$ and $G_3$ are either complete graphs of the same order or isolate vertices. If none of them is an isolate vertex, then $G_1\cong G_2\cong G_3\cong K_{n/3}$. It follows that $G=\overline{3K_{n/3}}=K_{\frac{n}{3},\frac{n}{3},\frac{n}{3}}$. If there's exactly one of them is an isolate vertex, say $G_3$, then $G_1\cong G_2\cong K_{(n-1)/2}$. It follows that $G=\overline{2K_{(n-1)/2}\cup K_1}=K_1\nabla K_{\frac{n-1}{2},\frac{n-1}{2}}$. If there are exactly two of them are isolate vertices, say $G_2$ and $G_3$, then $G_1\cong K_{n-2}$. It follows that $G=\overline{K_{n-2}\cup 2K_1}=K_2\nabla (n-2)K_1$. Conversely, note that all of $K_2\nabla (n-2)K_1$, $K_1\nabla K_{\frac{n-1}{2},\frac{n-1}{2}}$ and $K_{\frac{n}{3},\frac{n}{3},\frac{n}{3}}$ are the join of two graphs, by Lemma \ref{lem-2-1} (iv) and Theorem \ref{thm-2-2}, we obtain their distance Laplacian spectra, which are shown in (\ref{222}). Therefore, $K_2\nabla (n-2)K_1, K_1\nabla K_{\frac{n-1}{2},\frac{n-1}{2}}, K_{\frac{n}{3},\frac{n}{3},\frac{n}{3}}\in\mathcal{H}_2(n)$, and the result follows.
\end{proof}
Recall that $\mathcal{G}(n,n-3)=\mathcal{H}_1(n)\cup\mathcal{H}_2(n)$. Combining Corollary \ref{cor-3-2} and Lemma \ref{lem-3-10}, we completely determine $\mathcal{G}(n,n-3)$ in the following result.
\begin{thm}\label{thm-3-3}For an integer $n\ge 6$, we have
\[\mathcal{G}(n,n-3)=\{K_{2,n-2},K_{1,n-1}+e,K_{n/2,n/2}+e,K_2\nabla (n-2)K_1, K_1\nabla K_{\frac{n-1}{2},\frac{n-1}{2}},K_{\frac{n}{3},\frac{n}{3},\frac{n}{3}}\}.\]
\end{thm}
\begin{remark}
By using the software SageMath, we get the graphs with $m(\partial^L_1)=n-3$ for $n=4$ and $n=5$. That is,
\[\left\{\begin{array}{l}\mathcal{G}(4,1)=\{P_4,K_{1,3}+e,K_2\nabla 2K_1\}\\
\mathcal{G}(5,2)=\{K_{2,3},K_{1,4}+e,K_2\nabla 3K_1,K_1\nabla K_{2,2},C_5\} \end{array}\right..\]
\end{remark}
We end up this paper by the following result.
\begin{cor}\label{cor-3-3}
Let $G\in\mathcal{G}(n,n-3)$ with $n\ge 5$ then $G$ is determined by its distance Laplacian spectrum.
\end{cor}
\begin{proof}
Let $H\in\mathcal{G}(n)$ with $\emph{Spec}_{\mathcal{L}}(H)=\emph{Spec}_{\mathcal{L}}(G)$. We get that $H\in \mathcal{G}(n,n-3)$. Then, the result follows by pairwise comparing the distance Laplacian spectra of graphs in $\mathcal{G}(n,n-3)$, which are presented in (\ref{111}) and (\ref{222}).
\end{proof}


\begin{thebibliography}{11}{\small 
\bibitem{Aouchiche} M. Aouchiche, P. Hansen, Distance spectra of graphs: A survey, Linear Algebra Appl. 458 (2014) 301--384.\vspace{-0.3cm}
\bibitem{Aouchiche-2} M. Aouchiche, P. Hansen, Two Laplacians for the distance matrix of a graph, Linear Algebra Appl 439 (2013) 21--33.\vspace{-0.3cm}
\bibitem{Aouchiche-3} M. Aouchiche, P. Hansen, A Laplacian for the distance matrix of a graph, Czechoslovak Mathematical Journal 64 (2014) 751--761.\vspace{-0.3cm}
\bibitem{Cvet} D. Cvetkovi\'c, P. Rowlinson, S. Simi\'c, An Introduction to the Theory of Graph Spectra, Cambridge Uni. Press, New York, 2010.\vspace{-0.3cm}
\bibitem{Corneil} D. Corneil, H. Lerchs, L. Burlingham. Complement reducible graphs, Discrete Appl. Math. 3 (1981) 163--174.\vspace{-0.3cm}
\bibitem{Celso} Celso M. Da Silva JR., et al., A note on a conjecture for the distance Laplacian matrix, Electr. J. Linear Algbra 31 (2016) 60--68.\vspace{-0.3cm}
\bibitem{Dam} van Dam, E.R., Haemers, W.H., Graphs with constant $\mu$ and $\bar{\mu}$, Discrete Math. 182 (1998) 293--307.\vspace{-0.3cm}
\bibitem{Horn} R.A. Horn, C.R. Johnson, Matrix Analysis, Cambridge Univ. Press, New York, 1992.\vspace{-0.3cm}
\bibitem{Mohammadian} A. Mohammadian, B.Tayfeh-Rezaie, Graphs with four distinct Laplacian eigenvalues, J. Algebra Comb. 34 (2011) 671--682.\vspace{-0.3cm}
\bibitem{Nath} M. Nath, S. Paul, On the distance Laplacian spectra of graphs, Linear Algebra Appl 460 (2014) 97--110.\vspace{-0.3cm}
\bibitem{Tian} F. Tian, D. Wong, J. Rou, Proof for four conjectures about the distance Laplacian and distance signless Laplacian eigenvalues of a graph, Linear Algebra Appl 471 (2015) 10--20.}
\end{thebibliography}
\end{document}